%% file: Hilbert-rev.tex
\documentclass{amsart}
\usepackage{amssymb,latexsym,amsmath,amscd,graphicx,graphics,epic,eepic,bm,color,array,mathrsfs,fullpage}
\usepackage{enumerate}
\usepackage[all,knot,poly]{xy}

\newcommand{\zed}{\mathbb{Z}}

\newcommand{\Q}{\mathbb{Q}}
\newcommand{\fH}{\mathcal{H}}
\newcommand{\fC}{\mathcal{C}}

\newcommand{\bn}[2]{\genfrac{(}{)}{0pt}{}{#1}{#2}}

\newcommand{\ve}{\varepsilon}

\newcommand{\id}{\mathrm{id}}

\theoremstyle{plain}
\newtheorem{theorem}{Theorem}[section]
\newtheorem{lemma}[theorem]{Lemma}
\newtheorem{proposition}[theorem]{Proposition}
\newtheorem{corollary}[theorem]{Corollary}

\newtheorem{question}[theorem]{Question}

\theoremstyle{definition}
\newtheorem{definition}[theorem]{Definition}
\newtheorem{example}[theorem]{Example}

\newtheorem{acknowledgments}{Acknowledgments\ignorespaces}

\theoremstyle{remark}
\newtheorem{remark}[theorem]{Remark}

\numberwithin{equation}{section}

\begin{document}

\title{On the Hilbert Polynomial of the HOMFLYPT Homology}

\author{Hao Wu}

\thanks{The author was partially supported by NSF grant DMS-1205879.}

\address{Department of Mathematics, The George Washington University, Phillips Hall, Room 739, 801 22nd Street NW, Washington DC 20052, USA. Telephone: 1-202-994-0653, Fax: 1-202-994-6760}

\email{haowu@gwu.edu}



\begin{abstract}
We prove that the degree of the Hilbert polynomial of the HOMFLYPT homology of a closed braid $B$ is $l-1$, where $l$ is the number of components of $B$. This controls the growth of the HOMFLYPT homology with respect to its polynomial grading. The Hilbert polynomial also reveals a link polynomial hidden in the HOMFLYPT polynomial.
\end{abstract}

\maketitle

\section{Introduction}\label{sec-intro}
 
The HOMFLYPT homology was introduced by Khovanov and Rozansky in \cite{KR2}. In the current paper, we study its Hilbert polynomial. In the following, we use the $\fH_0$ normalization of the HOMFLYPT homology in \cite{Wu-triple-trans}, which is slightly more symmetric than the original normalization in \cite{KR2}. We will review $\fH_0$ in Section \ref{sec-H-0} below. 

$\fH_0$ has three $\zed$-gradings:
\begin{itemize}
	\item the homological grading with degree function $\deg_h$,
	\item the $a$-grading with degree function $\deg_a$,
	\item the $x$-grading with degree function $\deg_x$.
\end{itemize}
$\fH_0$ also has a $\zed_2$-grading. But this $\zed_2$-grading is always equal to the parity of the $a$-grading. Denote by $\fH_0^{i,j,k}$ the homogeneous component of $\fH_0$ with homological grading $i$, $a$-grading $j$ and $x$-grading $k$. From the definition, we know that $\fH_0^{i,j,k}=0$ unless $j+k$ is even. We call the sum of the $a$-grading and $x$-grading the polynomial grading of $\fH_0$. Its degree function is $\deg_a+\deg_x$. The homogeneous component of $\fH_0$ with homological grading $i$, $a$-grading $j$ and polynomial grading $T$ is $\fH_0^{i,j,T-j}$, which is zero unless $T$ is even. The following lemma is a special case of Hilbert's Syzygy Theorem (Theorem \ref{thm-syzygy} below.)
 
\begin{lemma}\label{lemma-def-Hilbert-H-0}
For a closed braid $B$ and a pair $(i,j) \in \zed\times\zed$, there is a polynomial $P_{B,i,j}(T) \in \Q[T]$ such that $\dim_\Q \fH_0^{i,j,2T-j}(B) = P_{B,i,j}(T)$ for all large positive integer $T$. In other words, $P_{B,i,j}(T)$ is the Hilbert polynomial of the direct sum $\bigoplus_{T\in\zed} \fH_0^{i,j,2T-j}(B)$ with respect to the polynomial grading.

Since the homological and $a$-gradings of $\fH_0(B)$ are both bounded, $P_{B,i,j}(T)$ is the zero polynomial for all but finitely many pairs of $i$ and $j$. Thus, the Hilbert polynomial of $\fH_0(B)$ is the sum $P_B(T)=\sum_{(i,j) \in \zed\times\zed}P_{B,i,j}(T)$.
\end{lemma}

The following is the main result of this paper.

\begin{theorem}\label{thm-Hilbert-degree}
Let $B$ be a closed braid with $l$ components. Then $P_B(T)$ is a polynomial of degree $l-1$. 
\end{theorem}

There are two main ingredients in the proof of Theorem \ref{thm-Hilbert-degree}. First, an argument by Rasmussen in \cite{Ras-2-bridge} shows that $\fH_0(B)$ is a finitely generated module over the polynomial ring generated by the components of $B$. This implies that the degree of $P_B(T)$ is at most $l-1$. Second, the computation tree argument by Franks and Williams in \cite{FW} can be used to show that the degree of $P_B(T)$ is at least $l-1$.

For knots, we have the following corollary.

\begin{corollary}\label{cor-Hilbert-knot}
Let $K$ be a knot. Then there are non-negative integers $D_{i,j}$ such that 
\begin{itemize}
	\item $D_{i,j}=0$ for all but finitely many pairs of integers $i$ and $j$,
	\item $\dim_\Q \fH_0^{i,j,2T-j}(K) = D_{i,j}$ for large $T$,
	\item $\sum_{(i,j) \in \zed\times\zed} D_{i,j}$ is an odd number.
\end{itemize}
\end{corollary}

\begin{figure}[ht]
\[
\xymatrix{
\input{B+} && \input{B-} && \input{B0}
}
\]
\caption{}\label{fig-skein}

\end{figure}

Let $F_B(\alpha,\xi)$ be the decategorification of $\fH_0(B)$. That is, 
\begin{equation}\label{eq-def-HOMFLYPT}
F_B(\alpha,\xi)=\sum_{(i,j,k)\in \zed^3} (-1)^i \alpha^j\xi^k\dim_\Q \fH_0^{i,j,k}(B)\in\Q[\alpha^{-1},\alpha,\xi^{-1}][[\xi]].
\end{equation}
This is the HOMFLYPT polynomial with the normalization:
\begin{equation}\label{eq-HOMFLYPT-normalization}
\begin{cases}
F_B(\alpha,\xi) \text{ is invariant under transverse Markov moves,} \\
\alpha^{-1} F_{B_+}(\alpha,\xi) - \alpha F_{B_-}(\alpha,\xi) = (\xi^{-1}-\xi)F_{B_0}(\alpha,\xi), \\
F_{B'}(\alpha,\xi)=-\alpha^{-1}\xi^{-1}F_B(\alpha,\xi), \\
F_U(\alpha,\xi) = \frac{\alpha^{-1}}{\xi^{-1}-\xi},
\end{cases}
\end{equation}
where 
\begin{itemize}
  \item transverse Markov moves are reviewed in Subsection \ref{subsec-trans} below,
	\item $B_+$, $B_-$ and $B_0$ are closed braids identical outside the part shown in Figure \ref{fig-skein},
	\item $B'$ is obtained from the closed braid $B$ by a negative stabilization,
	\item $U$ is the unknot with no crossings.
\end{itemize}

Consider the polynomial $F_B(\alpha\xi,\xi)$. The power of $\xi$ in this polynomial corresponds to the polynomial grading of $\fH_0(B)$ and is, therefore, always even. Thus, we can expand $F_B(\alpha\xi,\xi)$ as 
\begin{equation}\label{eq-c-T}
F_B(\alpha\xi,\xi) = \sum_{T\in \zed} c_{B,T}(\alpha) \xi^{2T},
\end{equation} 
where $c_{B,T}(\alpha) \in \zed[\alpha^{-1},\alpha]$ and $c_{B,T}(\alpha)=0$ if $T\ll -1$. A byproduct of our work is that, for $T\gg1$, $c_{B,T}(\alpha)$ is a polynomial of $\alpha$ and $T$.

\begin{proposition}\label{prop-hidden-polynomial}
Define $Q_B(\alpha,T)\in\Q[\alpha,\alpha^{-1},T]$ by $Q_B(\alpha,T):= \sum_{(i,j)\in \zed\times\zed} (-1)^i\alpha^j P_{B,i,j}(T)$. Then
\begin{enumerate}
	\item For $T\gg1$, $Q_B(\alpha,T)=c_{B,T}(\alpha)$.
	\item $Q_B(\alpha,T)$ is invariant under transverse Markov moves.
	\item $Q_B(\alpha,T)$ satisfies the skein relation: 
	\[
	\begin{cases}
	\alpha^{-1} Q_{B_+}(\alpha, T+1) - \alpha Q_{B_-}(\alpha, T) = Q_{B_0}(\alpha, T+1)- Q_{B_0}(\alpha, T),\\
	Q_{B'}(\alpha, T) = -\alpha^{-1} Q_{B}(\alpha, T+1), \\
	Q_{U}(\alpha, T) = \alpha^{-1}, \\
	\end{cases}
	\]
	where 
\begin{itemize}
	\item $B_+$, $B_-$ and $B_0$ are closed braids identical outside the part shown in Figure \ref{fig-skein},
	\item $B'$ is obtained from the closed braid $B$ by a negative stabilization,
	\item $U$ is the unknot with no crossings.
\end{itemize}
	\item Using (2) and (3), $Q_B(\alpha,T)$ can be computed by any transverse computation tree\footnote{See Definition \ref{def-computation-tree} and Theorem \ref{thm-computation-tree-exists} below.} for $B$.
\end{enumerate}
\end{proposition}

From the proof of Theorem \ref{thm-Hilbert-degree}, we also have the following simple corollary.

\begin{corollary}\label{cor-degree-hidden-polynomial}
Let $B$ be a closed braid with $l$ components. Then $\deg_T Q_B(\alpha,T) = l-1$.
\end{corollary}

\begin{example}\label{example-2-braids}
For $k\geq 0$, define $B_k$ to be the closed braid of two strands with $k$ positive crossings. For $k<0$, define $B_k$ to be the closed braid of two strands with $-k$ negative crossings. Then, for all $k \in \zed$,
\begin{eqnarray}
\label{eq-2-braids-even} Q_{B_{2k}}(a,T) & = & \alpha^{2k-1}(1+\alpha^{-1})(T-k), \\
\label{eq-2-braids-odd} Q_{B_{2k+1}}(a,T) & = & k\alpha^{2k} +(k+1)\alpha^{2k-1}.
\end{eqnarray}
Note that $\deg_T Q_{B_{2k}}(a,T) =1$ and $\deg_T Q_{B_{2k-1}}(a,T) =0$, which are what we expect from Corollary \ref{cor-degree-hidden-polynomial}.
\end{example}

\begin{question}
For a given closed braid $B$, what is the smallest $T_0 \in \zed$ such that $Q_B(\alpha,T)=c_{B,T}(\alpha)$ for all $T \geq T_0$?
\end{question}

The rest of this paper is organized as following:
\begin{itemize}
	\item In Section \ref{sec-Hilbert}, we recall the definition of the Hilbert polynomial.
	\item In Section \ref{sec-H-0}, we review the $\fH_0$ normalization of the HOMFLYPT homology and use Rasmussen's argument to show that the degree of the Hilbert polynomial of $\fH_0(B)$ is at most $l-1$.
	\item In Section \ref{sec-computation-tree}, we use Franks and Williams' computation tree argument to show that the degree of the Hilbert polynomial of $\fH_0(B)$ is at least $l-1$. Theorem \ref{thm-Hilbert-degree} through Example \ref{example-2-braids} above are all proved in this section.
\end{itemize}

\begin{acknowledgments}
I would like to thank the referee for carefully reading a previous draft of this manuscript, providing many valuable comments and finding quite a few mistakes and typos. 
\end{acknowledgments}

\section{The Hilbert polynomial}\label{sec-Hilbert}

In this section, we recall the definition of the Hilbert polynomial. For the convenience of use later on, we adopt the slightly unusual grading convention introduced by Khovanov and Rozansky in our formulation.

\begin{definition}\label{def-graded-module}
Let $R$ be the polynomial ring $R=\Q[x_1,\dots,x_l]$ graded so that each $x_i$ is homogeneous of degree $2$. 
\begin{itemize}
	\item A graded module $M$ over $R$ is an $R$-module $M$ with a grading $M=\bigoplus_{n\in \zed} M_n$ such that $x_i\cdot M_n \subset M_{n+2}$. 
	\item We say the grading of $M$ is even if $M_n=0$ whenever $n$ is odd. 
	\item For an integer $k$, $M\{k\}$ is $M$ with grading shifted by $k$. That is, $M_n = M\{k\}_{n+k}$, where $M\{k\}_{n+k}$ is the homogeneous component of $M\{k\}$ of degree $n+k$.
	\item A finitely generated free graded module over $R$ is a finitely generated graded module over $R$ that is free and admits a homogeneous $R$-basis, in other words, a graded module over $R$ that is isomorphic to a direct sum $\bigoplus_{j=1}^m R\{k_j\}$, where $k_j\in \zed$.
\end{itemize} 
\end{definition}

Hilbert's Syzygy Theorem is key to the existence of the Hilbert polynomial.

\begin{theorem}[Hilbert's Syzygy Theorem]\label{thm-syzygy}
Let $R$ be the polynomial ring $R=\Q[x_1,\dots,x_l]$ graded as in Definition \ref{def-graded-module}. Assume that $M=\oplus_{n\in \zed} M_{2n}$ is a finitely generated graded $R$-module whose grading is even. Then there is an exact sequence of graded $R$ modules $0 \rightarrow F_l \rightarrow F_{l-1} \rightarrow \cdots \rightarrow F_1 \rightarrow F_0 \rightarrow M \rightarrow 0$, in which,
\begin{itemize}
	\item each $F_j$ is a finitely generated free graded module over $R$ whose grading is even,
	\item each arrow is a homogeneous map of graded $R$-modules preserving the grading.
\end{itemize}

As a standard consequence, there is a polynomial $P(T)\in \Q[T]$ of degree at most $l-1$ such that $\dim_\Q M_{2T} = P(T)$ for $T\gg1$. This $P(T)$ is called the Hilbert polynomial of $M$.
\end{theorem}

The existence of the free resolution in Theorem \ref{thm-syzygy} is the traditional content of Hilbert's Syzygy Theorem, a detailed elementary proof of which can be found in for example \cite[Theorem 4.3]{Arrondo-notes}. The existence of the Hilbert polynomial in Theorem \ref{thm-syzygy} is a standard consequence of Hilbert's Syzygy Theorem. For a simple proof of this, see for example \cite[Corollary 4.4]{Arrondo-notes}.

\section{The $\fH_0$ normalization of the HOMFLYPT homology}\label{sec-H-0}

Next, we review the HOMFLYPT homology with the $\fH_0$ normalization. Roughly speaking, to each MOY graph, we associate a $\zed_2\times\zed^{2}$-graded matrix factorization. Each closed braid diagram $B$ is resolved into a collection of MOY graphs. Using the crossing information of $B$, we assemble the matrix factorizations of the MOY resolutions of $B$ into a chain complex $\fC_0(B)$ of matrix factorizations. The HOMFLYPT homology $\fH_0(B)$ is then defined from $\fC_0(B)$.

\subsection{$\zed_2\times\zed^{2}$-graded matrix factorizations}

Let $R$ be the polynomial ring $R=\Q[a,x_1,\dots,x_n]$. We define a $\zed^{2}$-grading on $R$ so that $a,x_1,\dots,x_n$ are homogeneous with $\deg a = (2,0)$ and $\deg x_i =(0,2)$ for $i=1,\dots, n$. A $\zed^{2}$-graded $R$-module $M$ is a $R$-module $M$ equipped with a $\zed^{2}$-grading such that, for any homogeneous element $m$ of $M$, $\deg (am) = \deg m + (2,0)$ and $\deg (x_i m) = \deg m +(0,2)$ for $i=1,\dots, k$. We call the first component of this $\zed^{2}$-grading of $M$ the $a$-grading and denote its degree function by $\deg_a$. We call the second component of this $\zed^{2}$-grading of $M$ the $x$-grading and denote its degree function by $\deg_x$. 

For a $\zed^{2}$-graded $R$-module $M$, we denote by $M\{j,k\}$ the $\zed^{2}$-graded $R$-module obtained by shifting the $\zed^{2}$-grading of $M$ by $(j,k)$. That is, for any homogeneous element $m$ of $M$, $\deg_{M\{j,k\}} m = \deg_M m + (j,k)$.

\begin{definition}\label{def-mf}
Let $w$ be a homogeneous element of $R=\Q[a,x_1,\dots,x_n]$ with degree $(2,2)$. A $\zed_2\times\zed^{2}$-graded matrix factorization $M$ of $w$ over $R$ is a collection of two $\zed^{2}$-graded free $R$-modules $M_0$, $M_1$ and two homogeneous $R$-module maps $d_0:M_0\rightarrow M_1$, $d_1:M_1\rightarrow M_0$ of degree $(1,1)$, called differential maps, such that 
\[
d_1 \circ d_0=w\cdot\id_{M_0}, \hspace{1cm}  d_0 \circ d_1=w\cdot\id_{M_1}.
\]
We usually write $M$ as $M_0 \xrightarrow{d_0} M_1 \xrightarrow{d_1} M_0$.

The $\zed_2$-grading of $M$ takes value $\ve$ on $M_\ve$. The $a$- and $x$-gradings of $M$ are the $a$- and $x$-gradings of the underlying $\zed^{2}$-graded $R$-module $M_0 \oplus M_1$.

Following \cite{KR1}, we denote by $M\left\langle 1\right\rangle$ the matrix factorization $M_1 \xrightarrow{d_1} M_0 \xrightarrow{d_0} M_1$
\end{definition}

Note that:
\begin{itemize}
	\item For any $\zed_2\times\zed^{2}$-graded matrix factorization $M$ of $w$ over $R$ and $j,k \in \zed$, $M\{j,k\}$ is naturally a $\zed_2\times\zed^{2}$-graded matrix factorization of $w$ over $R$.
	\item For any two $\zed_2\times\zed^{2}$-graded matrix factorizations $M$ and $M'$ of $w$ over $R$, $M\oplus M'$ is naturally a $\zed_2\times\zed^{2}$-graded matrix factorization of $w$ over $R$.
	\item Let $w$ and $w'$ be two homogeneous elements of $R$ with degree $(2,2)$. For $\zed_2\times\zed^{2}$-graded matrix factorizations $M$ of $w$ and $M'$ of $w'$ over $R$, the tensor product $M\otimes_R M'$ is the $\zed_2\times\zed^{2}$-graded matrix factorization of $w+w'$ over $R$ such that:
\begin{itemize}
	\item $(M\otimes M')_0 = (M_0\otimes M'_0)\oplus (M_1\otimes M'_1)$, $(M\otimes M')_1 = (M_1\otimes M'_0)\oplus (M_1\otimes M'_0)$,
	\item The differential is given by the signed Leibniz rule. That is, for $m\in M_\ve$ and $m'\in M'$, $d(m\otimes m')=(d_M m)\otimes m' + (-1)^\ve m \otimes (d_{M'} m')$.
\end{itemize}
\end{itemize}

\begin{definition}\label{def-morph-mf}
Let $w$ be a homogeneous element of $R$ with degree $(2,2)$, and $M$, $M'$ any two $\zed_2\times\zed^{2}$-graded matrix factorizations of $w$ over $R$. A homogeneous morphism from $M$ to $M'$ of degree $(j,k)$ is a homogeneous homomorphism of $\zed^2$-graded $R$-modules $f:M\rightarrow M'$ of degree $(j,k)$ preserving the $\zed_2$-grading and satisfying $d_{M'}f=fd_{M}$. 

Two homogeneous morphisms $f,g:M\rightarrow M'$ of degree $(j,k)$ of $\zed_2\oplus\zed^{2}$-graded matrix factorizations are called homotopic if there is an $R$-module homomorphism $h:M\rightarrow M'\left\langle 1\right\rangle$ of degree $(j-1,k-1)$ preserving the $\zed_2$-grading such that $f-g = d_{M'}h+hd_M$.
\end{definition}

\begin{definition}\label{def-koszul-mf}
If $a_0,a_1\in R$ are homogeneous elements with $\deg a_0 +\deg a_1=(2,2)$, then denote by $(a_0,a_1)_R$ the $\zed_2\times\zed^{2}$-graded matrix factorization $R \xrightarrow{a_0} R\{1-\deg_a a_0,~1-\deg_x{a_0}\} \xrightarrow{a_1} R$ of $a_0a_1$ over $R$. More generally, if $a_{1,0},a_{1,1},\dots,a_{l,0},a_{l,1}\in R$ are homogeneous with $\deg a_{j,0} +\deg a_{j,1}=(2,2)$, then denote by 
\[
\left(%
\begin{array}{cc}
  a_{1,0}, & a_{1,1} \\
  a_{2,0}, & a_{2,1} \\
  \dots & \dots \\
  a_{l,0}, & a_{l,1}
\end{array}%
\right)_R
\]
the tensor product $(a_{1,0},a_{1,1})_R \otimes_R (a_{2,0},a_{2,1})_R \otimes_R \cdots \otimes_R (a_{l,0},a_{l,1})_R$. This is a $\zed_2\times\zed^{2}$-graded matrix factorization of $\sum_{j=1}^l a_{j,0} a_{j,1}$ over $R$, and is call the Koszul matrix factorization associated to the above matrix.
\end{definition}

Note that Koszul matrix factorizations are finitely generated over $R$. The following are several lemmas about Koszul matrix factorizations that will be useful later on.

\begin{lemma}\cite[Proposition 2]{KR1}\label{entries-null-homotopic}
Let $a_{1,0},a_{1,1},\dots,a_{l,0},a_{l,1}$ be as in Definition \ref{def-koszul-mf} and
\[
M = \left(%
\begin{array}{ll}
  a_{1,0}, & a_{1,1} \\
  a_{2,0}, & a_{2,1} \\
  \dots & \dots \\
  a_{l,0}, & a_{l,1}
\end{array}%
\right)_R.
\]
If $r$ is an element of the ideal $(a_{1,0}, a_{1,1}, \dots, a_{l,0}, a_{l,1})$ of $R$, then the multiplication by $r$, as an endomorphism of $M$, is homotopic to $0$.
\end{lemma}

Lemma \ref{lemma-contraction-weak} below is essentially \cite[Proposition 9]{KR1}. For a proof of this lemma, see for example \cite[Proposition 2.17]{Wu-triple-trans}.

\begin{lemma}\cite[Proposition 9]{KR1}\label{lemma-contraction-weak}
Let $I$ be an ideal of $R$ generated by homogeneous elements. Assume $w$, $a_0$ and $a_1$ are homogeneous elements of $R$ such that $\deg w=\deg a_0 +\deg a_1 = (2,2)$ and $w+a_0a_1 \in I$. Then $w \in I+(a_0)$. 

Let $M$ be a $\zed_2\oplus\zed^{2}$-graded matrix factorization of $w$ over $R$, and $\widetilde{M}=M \otimes_R (a_0,a_1)_R$. Then ${\widetilde{M}/I\widetilde{M}}$ and ${M/(I+(a_0))M}$are both $\zed_2 \oplus\zed^{2}$-graded chain complexes of $R$-modules.

If $a_0$ is not a zero-divisor in $R/I$, then there is an $R$-linear quasi-isomorphism 
\[
f:{\widetilde{M}/I\widetilde{M}} \rightarrow {(M/(I+(a_0))M)\left\langle 1\right\rangle \{1-\deg_a a_0, 1-\deg_x a_0 \}}
\] 
that preserves the $\zed_2\oplus\zed^{2}$-grading.
\end{lemma}

\subsection{Matrix factorization of MOY graphs}

\begin{figure}[ht]
\[
\xymatrix{
\input{vertex}
}
\]
\caption{}\label{fig-MOY-vertex}

\end{figure}

\begin{definition}\label{def-MOY}
An MOY graph $\Gamma$ is an embedding of a directed $4$-valent graph in the plane so that each vertex of $\Gamma$ looks like the one in Figure \ref{fig-MOY-vertex}.

A marking of an MOY graph $\Gamma$ consists of 
\begin{itemize}
	\item a finite collection of marked points in the interiors of the edges of $\Gamma$ such that each edge of $\Gamma$ contains at least one marked point,
	\item an assignment that assigns to each marked point a distinct variable.
\end{itemize}
\end{definition}

Fix an MOY graph $\Gamma$ and a marking of $\Gamma$. Assume that $x_1,\dots,x_n$ are the variables assigned to marked points in $\Gamma$. Let $R=\Q[a,x_1,\dots,x_n]$ with the $\zed^2$-grading so that $a,x_1,\dots,x_n$ are homogeneous and $\deg a =(2,0)$, $\deg x_1=\cdots=\deg x_n =(0,2)$. Cut $\Gamma$ at all of its marked points. We get a collection of pieces $\Gamma_1,\dots,\Gamma_m$, each of which is of one of the two types in Figure \ref{fig-MOY-pieces}.

\begin{figure}[ht]

\[
\xymatrix{
\input{arc_marked} && \input{wide-edge_marked}
}
\]

\caption{}\label{fig-MOY-pieces}

\end{figure}

\begin{itemize}
	\item If $\Gamma_q = \Gamma_{i;s}$ in Figure \ref{fig-MOY-pieces}, then $R_q = \Q[a,x_i,x_s]$ and $\fC_0(\Gamma_q) = (a,x_s-x_i)_{R_q}$.
	\item If $\Gamma_q = \Gamma_{i,j;s,t}$ in Figure \ref{fig-MOY-pieces}, then $R_q = \Q[a,x_i,x_j,x_s,x_t]$ and
	\[
	\fC_0(\Gamma_q) = \left(%
  \begin{array}{cc}
  a, & x_s+x_t-x_i-x_j \\
  0, & x_sx_t-x_ix_j
  \end{array}%
  \right)_{R_q}\{0,-1\}.
	\]
\end{itemize}

\begin{definition}\label{def-mf-MOY}
\[
\fC_0(\Gamma) = \bigotimes_{q=1}^m (\fC_0(\Gamma_q) \otimes_{R_q} R),
\]
where the big tensor product ``$\bigotimes_{q=1}^m$" is taken over the ring $R=\Q[a,x_1,\dots,x_n]$.

Note that $\fC_0(\Gamma)$ is a $\zed_2\times\zed^{2}$-graded matrix factorization of $0$.
\end{definition}

\begin{figure}[ht]
$
\xymatrix{
\input{crossing-1-1-res-0} \ar@<8ex>[rr]^{\chi^0} && \input{crossing-1-1-res-1} \ar@<-6ex>[ll]^{\chi^1}
}
$
\caption{}\label{def-chi-fig}

\end{figure}

The following lemma is established by Khovanov and Rozansky in \cite{KR2}. For a detailed proof in the present normalization, see \cite[Lemma 3.15]{Wu-triple-trans}.

\begin{lemma}\label{lemma-def-chi}
Let $\Gamma_0$ and $\Gamma_1$ be the graphs marked at their endpoints in Figure \ref{def-chi-fig}. Then there exist homogeneous morphisms of $\zed_2 \times \zed^{2}$-graded matrix factorizations $\fC_0(\Gamma_0) \xrightarrow{\chi^0} \fC_0(\Gamma_1)$ and $\fC_0(\Gamma_1) \xrightarrow{\chi^1} \fC_0(\Gamma_0)$ of degree $(0,1)$ satisfying:
\begin{enumerate}
  \item $\chi^0$ and $\chi^1$ are homotopically non-trivial,
	\item $\chi^1 \circ \chi^0 \simeq (x_j-x_s)\id_{\fC_0(\Gamma_0)}$ and $\chi^0 \circ \chi^1 \simeq (x_j-x_s)\id_{\fC_0(\Gamma_1)}$.
\end{enumerate}
\end{lemma}

\subsection{Chain complex of closed braid diagrams}

\begin{definition}\label{def-marking-braid}
For a closed braid $B$, an arc of $B$ is a part of $B$ that starts and ends at crossings and contains no crossing in its interior. A marking of $B$ consists of
\begin{itemize}
	\item a finite collection of marked points in the interiors of the arcs of $B$ such that each arc of $B$ contains at least one marked point,
	\item an assignment that assigns to each marked point a distinct variable.
\end{itemize}
\end{definition}

Let $B$ be a closed braid with a marking. Assume $x_1,\dots,x_n$ are all the variables assigned to marked points in $B$. The ring $R:=\Q[a,x_1,\dots,x_n]$ is graded so that $\deg a = (2,0)$ and $\deg x_j =(0,2)$ for $j=1,\dots,n$.

Cut $B$ at all of its marked points. This cuts $B$ into a collection $\{T_1,\dots,T_m\}$ of simple tangles, each of which is of one of the three types in Figure \ref{tangle-pieces-fig} and is marked only at its end points.

\begin{figure}[ht]
$
\xymatrix{
\input{arc} && \input{crossing+} && \input{crossing-} 
}
$
\caption{}\label{tangle-pieces-fig}

\end{figure}

If $T_q=A$, then $R_q =\Q[a,x_i,x_s]$ and $\fC_0(T_q)$ is the chain complex
\begin{equation}\label{eq-def-chain-arc}
\fC_0(A)= 0 \rightarrow \underbrace{\fC_0(A)}_{0} \rightarrow 0, 
\end{equation}
where the $\fC_0(A)$ on the right hand side is the matrix factorization associated to the MOY graph $A$, and the under-brace indicates the homological grading.

\begin{figure}[ht]
$
\xymatrix{
&& \input{crossing+}  \ar@<-12ex>[lld]_{0}\ar@<12ex>[rrd]^{+1} && \\
 \input{crossing-1-1-res-0}&&&& \input{crossing-1-1-res-1} \\
&& \input{crossing-} \ar[llu]_{0} \ar[rru]^{-1} &&
}
$
\caption{}\label{crossing-res-fig}

\end{figure}

If $T_q=C_\pm$, then $R_q = \Q[a,x_i,x_j,x_s,x_t]$ and $\fC_0(T_q)$ is the chain complex
\begin{eqnarray}
\label{eq-def-chain-crossing+} \fC_0(C_+) & = & 0 \rightarrow \underbrace{\fC_0(\Gamma_1)\left\langle 1\right\rangle\{1,0\}}_{-1} \xrightarrow{\chi^1} \underbrace{\fC_0(\Gamma_0)\left\langle 1\right\rangle\{1,-1\}}_{0} \rightarrow 0, \\
\label{eq-def-chain-crossing-} \fC_0(C_-) & = & 0 \rightarrow \underbrace{\fC_0(\Gamma_0)\left\langle 1\right\rangle\{-1,1\}}_{0} \xrightarrow{\chi^0} \underbrace{\fC_0(\Gamma_1)\left\langle 1\right\rangle\{-1,0\}}_{1} \rightarrow 0,
\end{eqnarray}
where the morphisms $\chi^0$ and $\chi^1$ are defined in Lemma \ref{lemma-def-chi} and, again, the under-braces indicate the homological gradings, which is also indicated by the labels on the arrows in Figure \ref{crossing-res-fig}.

\begin{definition}\label{def-homology-braid}
We define the chain complex $\fC_0(B)$ associated to $B$ to be 
\[
\fC_0(B) := \bigotimes_{q=1}^{m} (\fC_0(T_q)\otimes_{R_q} R),
\]
where the big tensor product ``$\bigotimes_{q=1}^{m}$" is taken over $R$. 

$\fC_0(B)$ is a chain complex of $\zed_2\times \zed^{2}$-graded matrix factorizations of $0$. It has two differentials: a differential $d_{mf}$ from its underlying matrix factorization structure and a differential $d_\chi$ from the local chain complexes \eqref{eq-def-chain-arc}-\eqref{eq-def-chain-crossing-}. These two differentials commute.

The HOMFLYPT homology $\fH_0(B)$ is defined to be $\fH_0(B):= H(H(\fC_0(B),d_{mf}),d_\chi)$. Since $d_{mf}$ and $d_\chi$ are both homogeneous with respect to the $\zed_2\times \zed^{3}$-grading of $\fC_0(B)$, $\fH_0(B)$ inherits this $\zed_2\times \zed^{3}$-grading. We call the three $\zed$-gradings the homological grading, the $a$-grading and the $x$-grading.
\end{definition}

\begin{theorem}\cite{KR2}\label{cor-homology-inv}
Up to isomorphism of $\zed_2\times \zed^{3}$-graded $\Q$-spaces, $\fH_0(B)$ is independent of the marking of $B$ and invariant under transverse Markov moves\footnote{Transverse Markov moves will be reviewed in Subsection \ref{subsec-trans} below.}.

If $B'$ is obtained from $B$ by a negative stabilization, then $\fH_0(B') \cong \fH_0(B)\|1\| \{-1,-1\}$, where $\|s\|$ means shifting the homological grading up by $s$.
\end{theorem}

\begin{remark}\label{remark-poly-grading-even}
Note that, for $\fC_0(B)$, the parity of its $a$-grading $=$ the parity of its $x$-grading $=$ its $\zed_2$-grading. The polynomial grading of $\fC_0(B)$ is the sum of the $a$-grading and the $x$-grading. So the polynomial grading of $\fC_0(B)$ is always even. Moreover, with the polynomial grading, $\fC_0(B)$ is a graded $R=\Q[a,x_1,\dots,x_n]$-module, whose grading is even. Since $d_{mf}$ and $d_\chi$ are both homogeneous $R$-modules homomorphisms, $\fH_0(B)$ is also a graded $R=\Q[a,x_1,\dots,x_n]$-module, whose polynomial grading is even.
\end{remark}

\subsection{Upper bound of the degree of the Hilbert polynomial of $\fH_0(B)$}

We are now ready to show that the degree of the Hilbert polynomial of $\fH_0(B)$ is at most $l-1$, where $l$ is the number of components of $B$. Our proof is based on \cite[Lemma 3.4]{Ras-2-bridge} by Rasmussen. For the convenience of the reader, we include the proof of this lemma for $\fH_0$ below.

\begin{lemma}\cite[Lemma 3.4]{Ras-2-bridge}\label{lemma-ras-components}
Let $B$ be a closed braid with a marking. If $x$ and $y$ are variables assigned to marked points on the same component of $B$, then multiplications by $x$ and $y$ induce the same endomorphism of $\fH_0(B)$.
\end{lemma}

\begin{proof} (Following \cite[Lemma 3.4]{Ras-2-bridge}.)
Assume $x_1,\dots,x_n$ are all the variables assigned to marked points in $B$. Let $R=\Q[a,x_1,\dots,x_n]$.

Cut $B$ at all of its marked points. This cuts $B$ into a collection $\{T_1,\dots,T_m\}$ of simple tangles, each of which is of one of the three types in Figure \ref{tangle-pieces-fig} and is marked only at its end points. It is easy to see that, to prove the lemma, we only need to prove it for the case when $x$ and $y$ are variables assigned to endpoints of the same component of a particular $T_q$. Without loss of generality, we assume that $x$ and $y$ are variables both assigned to endpoints $T_1$.

If $T_1=A$ in Figure \ref{tangle-pieces-fig}, then the lemma follows easily from Lemma \ref{entries-null-homotopic}.

Next assume $T_1$ is a crossing. The proofs for positive and negative crossings are very similar. We only give the proof for a positive crossing. 

Assume $T_1=C_+$ in Figure \ref{tangle-pieces-fig}. Denote by $\hat{d}_\chi$ the differential of the chain complex of matrix factorizations $\fC:=\bigotimes_{q=2}^{m} (\fC_0(T_q)\otimes_{R_q} R)$, where the big tensor product ``$\bigotimes_{q=2}^{m}$" is taken over $R$. According to chain complex \eqref{eq-def-chain-crossing+} of $C_+$, 
\begin{itemize}
	\item $\fC_0(B)= \fC' \oplus\fC''$, where $\fC'$ and $\fC''$ are chain complexes of matrix factorizations of $0$ given by
	\begin{itemize}
		\item $\fC' = \fC_0(\Gamma_1)\otimes_{R_1}\fC \left\langle 1\right\rangle\{1,0\}\|-1\|$,
		\item $\fC'' = \fC_0(\Gamma_0)\otimes_{R_1}\fC \left\langle 1\right\rangle\{1,-1\}$,
	\end{itemize}
	\item the differential $d_\chi$ of $\fC_0(B)$ is of the form $d\chi = d_{-1} + d_{-1,0} + d_0$, where $d_{-1},~d_{-1,0},~ d_0$ are chain maps given by
	\begin{itemize}
		\item $d_{-1} = -\id_{\fC_0(\Gamma_1)}\otimes \hat{d}_\chi:\fC' \rightarrow \fC'$,
		\item $d_{-1,0} = \chi^1 \otimes \id_{\fC}: \fC' \rightarrow \fC''$,
		\item $d_0 = \id_{\fC_0(\Gamma_0)}\otimes \hat{d}_\chi: \fC'' \rightarrow \fC''$.
	\end{itemize}
\end{itemize}
Define $d_{0,-1}$ to be the homomorphism $d_{0,-1} = \chi^0 \otimes \id_{\fC}: \fC'' \rightarrow \fC'$. Then $d_{0,-1} \circ d_0 = \chi^0 \otimes \hat{d}_\chi = - d_{-1} \circ d_{0,-1}$.

Now assume $\Lambda \in H(\fC_0(B),d_{mf})$ satisfies $d_\chi(\Lambda) =0$. Write $\Lambda = \Lambda' + \Lambda''$, where $\Lambda' \in H(\fC',d_{mf})$ and $\Lambda'' \in H(\fC'',d_{mf})$. Then $d_{-1} (\Lambda') = 0$ and $d_{-1,0} (\Lambda') + d_0 (\Lambda'')=0$. Now consider $d_{0,-1}(\Lambda'') \in H(\fC',d_{mf})$. We have
\begin{eqnarray*}
d\chi(d_{0,-1}(\Lambda'')) & = & d_{-1}(d_{0,-1}(\Lambda'')) + d_{-1,0}(d_{0,-1}(\Lambda'')) =-d_{0,-1} ( d_0(\Lambda'')) + d_{-1,0}(d_{0,-1}(\Lambda'')) \\
& = & d_{0,-1} ( d_{-1,0}(\Lambda')) + d_{-1,0}(d_{0,-1}(\Lambda'')).
\end{eqnarray*}
By Lemma \ref{lemma-def-chi}, $d_{0,-1} \circ d_{-1,0} = (x_j-x_s) \id_{\fC'}$ and $d_{-1,0}\circ d_{0,-1} = (x_j-x_s) \id_{\fC''}$. So 
\[
d\chi(d_{0,-1}(\Lambda'')) = (x_j-x_s) (\Lambda' + \Lambda'') = (x_j-x_s) \Lambda.
\]
This shows that the multiplication by $x_j-x_s$ on $\fH_0(B) = H(H(\fC_0(B),d_{mf}),d_\chi)$ is the zero endomorphism, which proves the lemma for the case $x=x_s$ and $y=x_j$. By Lemma \ref{entries-null-homotopic}, multiplication by $x_s+x_t-x_i-x_j$ is the zero endomorphism on $H(\fC_0(B),d_{mf})$. Thus follows the case when $x=x_t$ and $y=x_i$. This completes the proof of the lemma for $T_1=C_+$.
\end{proof}

\begin{corollary}\label{cor-Hilbert-at-most-l-1}
Let $B$ be a closed braid with $l$ components. Then the degree of the Hilbert polynomial $P_B(T)$ of $\fH_0(B)$ is at most $l-1$.
\end{corollary}

\begin{proof}
Fix a marking of $B$. Assume $x_1,\dots,x_n$ are all the variables assigned to marked points in $B$. Let $R:=\Q[a,x_1,\dots,x_n]$. Consider the ideal $I=(\{a\}\cup\{x_i-x_j~|~x_i,~x_j$ are assigned to the same component of $B.\})$ By Lemmas \ref{entries-null-homotopic} and \ref{lemma-ras-components}, multiplication by any element of $I$ is the zero endomorphism of $\fH_0(B)$. Fix a single marking $x_{i_k}$ on the $k$-th component of $B$. Then the $R$-action on $\fH_0(B)$ factors through the quotient ring $R/I \cong \Q[x_{i_1},\dots,x_{i_l}]$. Moreover, since $\fC_0(B)$ is finitely generated over $R$ and $R$ is Noetherian, $\fH_0(B)$ is a finitely generated $R$-module. Thus, $\fH_0(B)$ is a finitely generated $\Q[x_{i_1},\dots,x_{i_l}]$-module. Clearly, with the polynomial grading, $\fH_0(B)$ is a graded $\Q[x_{i_1},\dots,x_{i_l}]$-module, whose grading is even. By Hilbert's Syzygy Theorem (Theorem \ref{thm-syzygy}),  the degree of $P_B(T)$ is at most $l-1$.
\end{proof}

\section{The Computation Tree Argument}\label{sec-computation-tree}

In this section, we use Franks and Williams' computation tree argument to show that the degree of the Hilbert polynomial of $\fH_0(B)$ is at least $l-1$, where $l$ is the number of components of $B$. This will complete the proof of Theorem \ref{thm-Hilbert-degree}. 

\subsection{Transverse computation trees}\label{subsec-trans}

First, we recall the concept of transverse computation tree, which is defined in \cite{FW} and called ``invariant computation tree" there.

Recall that two closed braids represent the same smooth link if and only if one of them can be changed into the other by a finite sequence of Markov moves, which are:
\begin{itemize}
    \item Braid group relations generated by
    \begin{itemize}
	  \item $\sigma_i\sigma_i^{-1}=\sigma_i^{-1}\sigma_i=\emptyset$,
	  \item $\sigma_i\sigma_j=\sigma_j\sigma_i$, when $|i-j|>1$,
	  \item $\sigma_i\sigma_{i+1}\sigma_i=\sigma_{i+1}\sigma_i\sigma_{i+1}$.
    \end{itemize}
    \item Conjugations: $\mu\leftrightsquigarrow\eta^{-1}\mu\eta$,
    where $\mu,~\eta\in \mathbf{B}_m$.
    \item Stabilizations and destabilizations:
    \begin{itemize}
	  \item positive: $\mu~(\in \mathbf{B}_m)\leftrightsquigarrow \mu\sigma_m~(\in \mathbf{B}_{m+1})$,
	  \item negative: $\mu~(\in \mathbf{B}_m)\leftrightsquigarrow \mu\sigma_m^{-1}~(\in \mathbf{B}_{m+1})$.
    \end{itemize}
\end{itemize}
In the above, $\mathbf{B}_m$ is the braid group on $m$ strands.

The standard contact structure $\xi_{st}$ on $\mathbb{R}^3$ is the kernel of the contact form $\alpha_{st} = dz-ydx+xdy=dz+r^2d\theta$. An oriented smooth link $L$ in $\mathbb{R}^3$ is transverse if $\alpha_{st}|_L>0$. Two transverse links are said to be transversely isotopic if there is an isotopy from one to the other through transverse links. It is known that: 
\begin{itemize}
	\item Every transverse link is transverse isotopic to a counterclockwise transverse closed braid around the $z$-axis. (See \cite{Ben}.)
	\item Any smooth counterclockwise closed braid around the $z$-axis can be smoothly isotoped into a counterclockwise transverse closed braid around the $z$-axis without changing the braid word. (A simple observation.)
\end{itemize}
The following theorem by Orevkov, Shevchishin \cite{OSh} and Wrinkle \cite{Wr} describes when two transverse closed braids are transversely isotopic. 

\begin{theorem}\cite{OSh,Wr}\label{transversal-markov}
Two counterclockwise transverse closed braids around the $z$-axis are transversely isotopic if and only if the braid word of one of them can be changed into that of the other by a finite sequence of braid group relations, conjugations and positive stabilizations and destabilizations.
\end{theorem}

We call braid group relations, conjugations and positive stabilizations/destabilizations the transverse Markov moves. 

\begin{figure}[ht]
$
\xymatrix{ 
\input{Conway+} &&& \input{Conway-}
}
$
\caption{}\label{fig-Conway}

\end{figure}

Following \cite{FW}, a Conway splitting of a closed braid at a crossing consists of two choices of local changes to the braid, encoded diagrammatically as in Figure \ref{fig-Conway}, where
\begin{itemize}
	\item $B_+$, $B_-$ and $B_0$ are closed braids identical outside the part shown in Figure \ref{fig-skein},
	\item the crossing of $B_+$/$B_-$ shown in Figure \ref{fig-skein} is the crossing where the Conway splitting is performed at.
\end{itemize}

\begin{definition}\cite{FW}\label{def-computation-tree}
A transverse computation tree $\Upsilon$ is a connected, rooted, oriented binary tree, with each node labeled by a closed braid satisfying:
\begin{enumerate}
	\item If a non-terminal node $N$ of $\Upsilon$ is labeled by a closed braid $B$, then its two child nodes $N_0$ and $N_1$ are labeled by closed braids $B_0$ and $B_1$ obtained from $B$ by 
\begin{itemize}
	\item first performing a finite sequence of transverse Markov moves on $B$ to get a closed braid $\hat{B}$,
	\item then performing a Conway splitting at a crossing of $\hat{B}$ to get $B_0$ and $B_1$.
\end{itemize}
   \item If $N$ is a terminal node of $\Upsilon$, then $N$ is labeled by a closed braid with no crossings.
\end{enumerate}

\end{definition}

\begin{theorem}\cite[Theorem 1.7]{FW}\label{thm-computation-tree-exists}
For any closed braid $B$, there is a transverse computation tree whose root is labeled by $B$.
\end{theorem}

\subsection{Decategorification of $\fH_0(B)$}

Now, let us quickly explain the normalization of the decategorification of $\fH_0(B)$. Recall that the decategorification $F_B(\alpha,\xi)$ of $\fH_0(B)$ is defined in \eqref{eq-def-HOMFLYPT}. By Theorem \ref{cor-homology-inv}, we know that 
\begin{itemize}
	\item[(a)] $F_B(\alpha,\xi)$ is invariant under transverse Markov moves.
	\item[(b)] If $B'$ is obtained from $B$ by a negative stabilization, then $F_{B'}(\alpha,\xi)= -\alpha^{-1}\xi^{-1}F_B(\alpha,\xi)$.
\end{itemize}
From the definition of $\fH_0(B)$, especially the local chain complexes \eqref{eq-def-chain-crossing+} and \eqref{eq-def-chain-crossing-}, we know that
\begin{itemize}
	\item[(c)] $\alpha^{-1} F_{B_+}(\alpha,\xi) - \alpha F_{B_-}(\alpha,\xi) = (\xi^{-1}-\xi)F_{B_0}(\alpha,\xi)$, where $B_+$, $B_-$ and $B_0$ are closed braids identical outside the part shown in Figure \ref{fig-skein}.
\end{itemize}
Finally, for the unknot $U$ without any crossings, $\fH_0(U)$ is the homology of the Koszul matrix factorization $(a,0)_{\Q[x]}$. It is straightforward to check that $\fH_0(U) \cong \Q[x]\{-1,1\}$. Recall that $\deg x =(0,2)$. This implies that 
\begin{itemize}
	\item[(d)] $F_U(\alpha, \xi) = \alpha^{-1}\xi\sum_{n=0}^{\infty}\xi^{2n} = \frac{\alpha^{-1}}{\xi^{-1}-\xi}$.
\end{itemize}

Of course, by (a-d) above, we know that $F_B(\alpha,\xi)$ satisfies normalization \eqref{eq-HOMFLYPT-normalization}.

\begin{lemma}\label{lemma-unlink}
Denote by $U^{\sqcup l}$ the $l$-strand closed braid with no crossings. Then 
\[
F_{U^{\sqcup l}}(\alpha, \xi) =\alpha^{-1}\xi(1+\alpha^{-1}\xi)^{l-1}\sum_{T=0}^\infty \bn{T}{l-1} \xi^{2T}.
\]

In particular, $F_{U^{\sqcup l}}(\alpha\xi, \xi) = \alpha^{-1}(1+\alpha^{-1})^{l-1}\sum_{T=0}^\infty \bn{T}{l-1} \xi^{2T}$ and $Q_{U^{\sqcup l}}(\alpha, T) = \alpha^{-1}(1+\alpha^{-1})^{l-1} \bn{T}{l-1}$.
\end{lemma}

\begin{proof}
$\fH_0(U^{\sqcup l})$ is the homology of the Koszul matrix factorization 
\[
\fC_0(U^{\sqcup l}) = \left(%
\begin{array}{ll}
  a, & 0\\
  a, & 0\\
  \dots & \dots \\
  a, & 0
\end{array}%
\right)_{\Q[a,x_1,\dots,x_l]},
\]
where there are $l$-rows. In Lemma \ref{lemma-contraction-weak}, let $I=0$, $a_0=a$, $a_1=0$ and $\widetilde{M}=\fC_0(U^{\sqcup l})$. Then $\fH_0(U^{\sqcup l})$ is isomorphic to underlying $\zed_2\times\zed^3$-graded $\Q[a,x_1,\dots,x_l]$-module of the Koszul matrix factorization
\[
\left(%
\begin{array}{ll}
  0, & 0\\
  \dots & \dots \\
  0, & 0
\end{array}%
\right)_{\Q[x_1,\dots,x_l]} \left\langle 1\right\rangle \{-1, 1\},
\]
where there are $l-1$ rows, and $a$ acts on $\Q[x_1,\dots,x_l]$ as $0$.

From this, we have $F_{U^{\sqcup l}}(\alpha, \xi) = \alpha^{-1}\xi(1+\alpha^{-1}\xi)^{l-1}\sum_{T=0}^\infty \bn{T}{l-1} \xi^{2T}$ and $Q_{U^{\sqcup l}}(\alpha, T) = \alpha^{-1}(1+\alpha^{-1})^{l-1} \bn{T}{l-1}$.
\end{proof}

\subsection{Skein relation for $Q_B(\alpha,T)$}

In this subsection, we prove Proposition \ref{prop-hidden-polynomial}.

\begin{proof}[Proof of Proposition \ref{prop-hidden-polynomial}]
Recall that $F_B(\alpha,\xi)=\sum_{(i,j,k)\in \zed^3} (-1)^i \alpha^j\xi^k\dim_\Q \fH_0^{i,j,k}(B)$. So 
\begin{eqnarray*}
F_B(\alpha\xi,\xi) & = &\sum_{(i,j,k)\in \zed^3} (-1)^i \alpha^j\xi^{j+k}\dim_\Q \fH_0^{i,j,k}(B) \\
& = & \sum_{(i,j,T)\in \zed^3} (-1)^i \alpha^j\xi^{2T}\dim_\Q \fH_0^{i,j,2T-j}(B) \\
& = & \sum_{T\in \zed}(\sum_{(i,j)\in \zed^2} (-1)^i \alpha^j\dim_\Q \fH_0^{i,j,2T-j}(B))\xi^{2T}
\end{eqnarray*}
where, in the second step, we used the fact that the polynomial grading of $\fH_0(B)$ is even, that is, \linebreak $\fH_0^{i,j,2T+1-j}(B) \cong 0$ $\forall ~(i,j,T)\in \zed^3$. This shows that 
\[
c_{B,T}(\alpha) = \sum_{(i,j)\in \zed^2} (-1)^i \alpha^j\dim_\Q \fH_0^{i,j,2T-j}(B) ~\forall ~T\in \zed.
\]
But $\dim_\Q \fH_0^{i,j,2T-j}(B) = P_{B,i,j}(T)$ for all large positive integer $T$ by Lemma \ref{lemma-def-Hilbert-H-0}. And $\bigoplus_{T\in \zed}\fH_0^{i,j,2T-j}(B)\cong 0$ for all but finitely many $(i,j)\in \zed^2$. This implies that, for $T\gg1$, 
\[
c_{B,T}(\alpha)=\sum_{(i,j)\in \zed\times\zed} (-1)^i\alpha^j P_{B,i,j}(T)=Q_B(\alpha,T),
\]
which is part (1) of the proposition.

For part (2) of the proposition, note that, by normalization \eqref{eq-HOMFLYPT-normalization}, $F_B(\alpha,\xi)$ is invariant under transverse Markov moves. So, $c_{B,T}(\alpha)$ is invariant under these moves. By part (1), for $T\gg1$, $Q_B(\alpha,T)$ is invariant under these moves. But $Q_B(\alpha,T)\in\Q[\alpha,\alpha^{-1},T]$ is a polynomial. This implies that $Q_B(\alpha,T)$ is invariant under transverse Markov moves.

Using the normalization \eqref{eq-HOMFLYPT-normalization} of $F_B(\alpha,\xi)$, one can check that $c_{B,T}(\alpha)$ satisfies the skein relations in part (3) of the proposition. By part (1), for $T\gg1$, $Q_B(\alpha,T)$ satisfies the skein relations in part (3) of the proposition. But, again, $Q_B(\alpha,T)$ is a polynomial. So $Q_B(\alpha,T)$ satisfies the skein relations in part (3) of the proposition for all $T$.

Finally, we prove part (4) of the proposition. For a polynomial $g(\alpha,T)\in \Q[\alpha,\alpha^{-1},T]$, defined two operators $S_\alpha$ and $\Delta_\alpha$ by $S_\alpha(g)(\alpha,T)=\alpha^{-2}g(\alpha,T+1)$ and $\Delta_\alpha(g)(\alpha,T)=\alpha(g(\alpha,T)-g(\alpha,T-1))$. Note that:
\begin{itemize}
	\item $S_\alpha(g),~\Delta_\alpha(g) \in \Q[\alpha,\alpha^{-1},T]$.
	\item $S_\alpha$ is invertible and $S_\alpha^{-1} (g)(\alpha,T) = \alpha^{2}g(\alpha,T-1)$.
	\item $S_\alpha$ and $\Delta_\alpha$ commute with each other.
\end{itemize}
According to the skein relation of $Q_B(\alpha,T)$ from part (3), we know that
\begin{eqnarray}
\label{eq-Q-skein-alpha+} Q_{B_+}(\alpha,T) & = & S_\alpha^{-1} (Q_{B_-})(\alpha,T) + \Delta_\alpha(Q_{B_0})(\alpha,T), \\
\label{eq-Q-skein-alpha-} Q_{B_-}(\alpha,T) & = & S_\alpha (Q_{B_+})(\alpha,T) - S_\alpha(\Delta_\alpha(Q_{B_0}))(\alpha,T),
\end{eqnarray}
where $B_+$, $B_-$ and $B_0$ are closed braids identical outside the part shown in Figure \ref{fig-skein}.

\begin{figure}[ht]
$
\xymatrix{ 
\input{Conway+-labeled} &&& \input{Conway--labeled}
}
$
\caption{}\label{fig-Conway-labeled}

\end{figure}

Now assume that $\Upsilon$ is a transverse computation tree with its root labeled by the closed braid $B$. We label each edge of $\Upsilon$ according to the Conway splitting performed as in Figure \ref{fig-Conway-labeled}. Note that, in Figure \ref{fig-Conway-labeled}, the operations labeling the edges apply to the two children in a Conway splitting.

Recall that
\begin{itemize}
	\item each terminal node $N$ in $\Upsilon$ in labeled by a closed braid $B_N$ with no crossings,
	\item there is a unique path from the root to each terminal node $N$ in $\Upsilon$.
\end{itemize}
Denote by $\Psi_N$ the composition, from left to right, of all the operators labeling the edges in the path from the root to the terminal node $N$. Then
\begin{equation}\label{eq-tree-computation-Q}
Q_{B}(\alpha, T) = \sum_{N\text{ is a terminal node of }\Upsilon} \Psi_N(Q_{B_N})(\alpha, T).
\end{equation}
This shows that $Q_{B}(\alpha, T)$ can be computed from $\Upsilon$.
\end{proof}

Now, it is straightforward to adapt the proof of \cite[Proposition 1.11]{FW} to prove that the degree of the Hilbert polynomial of the HOMFLYPT homology $\fH_0(B)$ of a closed braid $B$ is at least $l-1$, where $l$ is the number of components of $B$.

\begin{lemma}\label{lemma-degree-Hilbert-lower-bound}
Let $B$ be a closed braid with $l$ components. Then $\deg_T Q_{B}(1, T) = l-1$.
\end{lemma}

\begin{proof}(Adapted from the proof of \cite[Proposition 1.11]{FW}.) 
Defined two operators $S=S_\alpha|_{\alpha=1}$ and $\Delta=\Delta_\alpha|_{\alpha=1}$. That is, for a polynomial $g(T)\in \Q[T]$, $S(g)(T)=g(T+1)$ and $\Delta(g)(T)=g(T)-g(T-1)$. Note that:
\begin{itemize}
	\item $S(g),~\Delta(g) \in \Q[T]$, and:
	\begin{itemize}
		\item $\deg_T S(g) = \deg_T g$, $\deg_T \Delta(g) = \deg_T g -1$;
		\item If $\deg_T g =n$ and the coefficient of $T^{n}$ in $g$ is $c$, then the coefficient of $T^{n}$ in $S(g)$ is $c$ and the coefficient of $T^{n -1 }$ in $\Delta(g)$ is $nc$.
	\end{itemize}
		\item $S$ is invertible and $S^{-1} (g)(T) = g(T-1)$.
	\item $S$ and $\Delta$ commute with each other.
\end{itemize}
From equations \eqref{eq-Q-skein-alpha+} and \eqref{eq-Q-skein-alpha-}, we have
\begin{eqnarray}
\label{eq-Q-skein-alpha+spec} Q_{B_+}(1,T) & = & S^{-1} (Q_{B_-})(1,T) + \Delta(Q_{B_0})(1,T), \\
\label{eq-Q-skein-alpha-spec} Q_{B_-}(1,T) & = & S (Q_{B_+})(1,T) - S(\Delta_\alpha(Q_{B_0}))(1,T),
\end{eqnarray}
where $B_+$, $B_-$ and $B_0$ are closed braids identical outside the part shown in Figure \ref{fig-skein}.

\begin{figure}[ht]
$
\xymatrix{ 
\input{Conway+-labeled-spec} &&& \input{Conway--labeled-spec}
}
$
\caption{}\label{fig-Conway-labeled-spec}

\end{figure}

Let $\Upsilon$ be any transverse computation tree with root labeled by $B$. We label each edge of $\Upsilon$ according to the Conway splitting performed as in Figure \ref{fig-Conway-labeled-spec}. Again, operations labeling edges apply to children in Conway splittings. By equation \eqref{eq-tree-computation-Q}, we get
\begin{equation}\label{eq-tree-computation-Q-spec}
Q_{B}(1, T) = \sum_{N\text{ is a terminal node of }\Upsilon} \Phi_N(Q_{B_N})(1, T),
\end{equation}
where 
\begin{itemize}
	\item $B_N$ is the closed braid with no crossings labeling $N$,
	\item $\Phi_N$ is the composition, from left to right, of all the operators labeling the edges in the path from the root to the terminal node $N$. 
\end{itemize}

Since $S$ and $\Delta$ commute with each other, for any terminal node $N$ of $\Upsilon$, $\Phi_N$ is of the form $\Phi_N = \pm S^{p_N}\circ \Delta^{q_N}$, where $p_N,q_N \in \zed$ and $q_N\geq0$. There is a unique terminal node $N_0$ of $\Upsilon$ such that $q_{N_0}=0$, which is obtained by always choosing the branch $B_{\pm} \leadsto B_{\mp}$ in each Conway splitting. Note that:
\begin{itemize}
	\item The change $B_{\pm} \leadsto B_{\mp}$ in a Conway splitting does not change the number of components in the closed braid.
	\item The change $B_{0} \leadsto B_{\pm}$ in a Conway splitting either increase the number of components by $1$ or decrease the number of components by $1$.
\end{itemize}
Denote by $l_N$ the number of components of the closed braid $B_N$ with no crossings. Then $l_{N_0} = l$ and $l_N-q_N \leq l$ for any terminal node $N$ of $\Upsilon$. 

By Lemma \ref{lemma-unlink}, $Q_{B_N}(1,T) = 2^{l_N-1}\bn{T}{l_N-1}$ for any terminal node $N$ of $\Upsilon$. From the labeling of edges of $\Upsilon$ in Figure \ref{fig-Conway-labeled-spec}, we know that $\Phi_{N_0}=S^{p_{N_0}}$ since the negative sign only appears in the change $B_- \leadsto B_0$. Thus, we know that: 
\begin{itemize}
	\item[(a)] $\deg_T \Phi_{N_0}(Q_{B_{N_0}}) = l-1$ and the coefficient of $T^{l-1}$ in $\Phi_{N_0}(Q_{B_{N_0}})$ is $\frac{2^{l-1}}{(l-1)!}$. 
\end{itemize}
One can also see that:
\begin{itemize}
	\item[(b)] For any terminal node $N$ of $\Upsilon$, $\deg_T \Phi_{N}(Q_{B_{N}}) = l_N-1-q_N \leq l-1$. 
\end{itemize}
Now assume $N$ is a terminal node $N$ of $\Upsilon$ such that $N \neq N_0$ and $\deg_T \Phi_{N}(Q_{B_{N}}) = l-1$. Then $l_N-q_N = l$, $q_N>0$ and $l_N = l+q_N \geq l+1$. The coefficient of $T^{l-1}$ in $\Phi_{N}(Q_{B_{N}})$ is 
\[
\pm\frac{2^{l_N-1}}{(l_N-1)!}\cdot (l_N-1)\cdot (l_N-2)\cdots (l_N-q_N)=\pm\frac{2^{l_N-1}}{(l_N-q_N-1)!} = \pm\frac{2^{l_N-1}}{(l-1)!}.
\]
This shows that:
\begin{itemize}
	\item[(c)]  If $N$ is a terminal node $N$ of $\Upsilon$ such that $N \neq N_0$ and $\deg_T \Phi_{N}(Q_{B_{N}}) = l-1$, then the coefficient of $T^{l-1}$ in $\Phi_{N}(Q_{B_{N}})$ is $\pm\frac{2^{l_N-1}}{(l-1)!}$, where $l_N-1 \geq l$. 
\end{itemize}
Combining equation \eqref{eq-tree-computation-Q-spec} and conclusions (a)-(c) above, we know that $\deg_T Q_{B}(1, T) \leq l-1$ and that the coefficient of $T^{l-1}$ in $Q_{B}(1, T)$ is of the form $\frac{k}{(l-1)!}$, where $k$ is an integer satisfying $k \equiv 2^{l-1} \mod 2^l$. In particular, $k \neq 0$. Thus, $\deg_T Q_{B}(1, T) = l-1$.
\end{proof}

We are now ready to prove Theorem \ref{thm-Hilbert-degree}, Corollaries \ref{cor-Hilbert-knot}, \ref{cor-degree-hidden-polynomial} and example \ref{example-2-braids}.

\begin{proof}[Proof of Theorem \ref{thm-Hilbert-degree}]
By Corollary \ref{cor-Hilbert-at-most-l-1}, $\deg_T P_B(T)\leq l-1$. Recall that $P_B(T) =\sum_{(i,j) \in \zed\times\zed}P_{B,i,j}(T)$ and $Q_B(T) =\sum_{(i,j) \in \zed\times\zed}(-1)^{i}P_{B,i,j}(T)$. The leading coefficient of $P_{B,i,j}(T)$ is non-negative since $P_{B,i,j}(T)$ is the Hilbert polynomial of $\bigoplus_{T\in\zed} \fH_0^{i,j,2T-j}(B)$. This shows that $\deg_T P_B(T) = \max \{\deg_T P_{B,i,j}(T)~|~(i,j) \in \zed\times\zed\}$. But, by Lemma \ref{lemma-degree-Hilbert-lower-bound}, $\deg_T Q_{B}(1, T) = l-1$. This implies that $\deg_T P_{B,i,j}(T) \geq l-1$ for some $(i,j) \in \zed\times\zed$. So $\deg_T P_B(T)\geq l-1$.
\end{proof}

\begin{proof}[Proof of Corollary \ref{cor-Hilbert-knot}]
Corollary \ref{cor-Hilbert-knot} follows from Theorem \ref{thm-Hilbert-degree} except for the statement that $\sum_{(i,j) \in \zed\times\zed} D_{i,j}$ is an odd number. But $Q_K(1,T)=\sum_{(i,j) \in \zed\times\zed} (-1)^{i} D_{i,j}$. According to the last paragraph of the proof of Lemma \ref{lemma-degree-Hilbert-lower-bound}, $\sum_{(i,j) \in \zed\times\zed} (-1)^{i} D_{i,j} \equiv 2^0 \mod 2^1$. This implies that $\sum_{(i,j) \in \zed\times\zed} D_{i,j}$ is an odd number.
\end{proof}

\begin{proof}[Proof of Corollary \ref{cor-degree-hidden-polynomial}]
Note that $\deg P_{B,i,j}(T) \leq \deg P_B(T) = l-1$ for all $(i,j) \in \zed\times \zed$. Since $Q_B(\alpha,T):= \sum_{(i,j)\in \zed\times\zed} (-1)^i\alpha^j P_{B,i,j}(T)$, this implies that $\deg_T Q_B(\alpha,T) \leq l-1$. But, by Lemma \ref{lemma-degree-Hilbert-lower-bound} $\deg_T Q_{B}(1, T) = l-1$. This implies that $\deg_T Q_B(\alpha,T) \geq l-1$. Thus, $\deg_T Q_B(\alpha,T) = l-1$.
\end{proof}

\begin{proof}[Proof of Example \ref{example-2-braids}]
We only prove equations \eqref{eq-2-braids-even} and \eqref{eq-2-braids-odd} for $k \geq 0$. The proof for $k<0$ is very similar and left to the reader.

Note that $B_0$ is the $2$-strand closed braid with no crossings and $B_1$ is the positive stabilization of the $1$-strand closed braid. So, by Lemma \ref{lemma-unlink}, we have that
\begin{eqnarray*}
Q_{B_0}(\alpha,T) & = & \alpha^{-1}(1+\alpha^{-1})T,\\
Q_{B_1}(\alpha,T) & = & \alpha^{-1}.
\end{eqnarray*}
This shows that equations \eqref{eq-2-braids-even} and \eqref{eq-2-braids-odd} are true for $k=0$. 

Now assume that \eqref{eq-2-braids-even} and \eqref{eq-2-braids-odd} are true for a given $k\geq 0$. For $B_n$ with $n>0$, applying the skein relation of $Q_B(\alpha, T)$ at a positive crossing of $B_n$, one gets
\begin{equation}\label{eq-2-braid-skein}
Q_{B_n}(\alpha, T) = \alpha^2 Q_{B_{n-2}}(\alpha, T-1) + \alpha (Q_{B_{n-1}}(\alpha, T) - Q_{B_{n-1}}(\alpha, T-1)).
\end{equation}
For $n=2k+2$, by the induction hypothesis, equation \eqref{eq-2-braid-skein} gives that
\begin{eqnarray*}
Q_{B_{2k+2}}(\alpha, T) & = & \alpha^2 Q_{B_{2k}}(\alpha, T-1) + \alpha (Q_{B_{2k+1}}(\alpha, T) - Q_{B_{2k+1}}(\alpha, T-1)) \\
& = & \alpha^2 Q_{B_{2k}}(\alpha, T-1) = \alpha^{2k+1}(1+\alpha^{-1})(T-k-1).
\end{eqnarray*}
This proves equation \eqref{eq-2-braids-even} for $k+1$. For $n=2k+3$, by the induction hypothesis and the above computation of $Q_{B_{2k+2}}(\alpha, T)$, equation \eqref{eq-2-braid-skein} gives that
\begin{eqnarray*}
Q_{B_{2k+3}}(\alpha, T) & = & \alpha^2 Q_{B_{2k+1}}(\alpha, T-1) + \alpha (Q_{B_{2k+2}}(\alpha, T) - Q_{B_{2k+2}}(\alpha, T-1)) \\
& = & \alpha^2(k\alpha^{2k} +(k+1)\alpha^{2k-1}) + \alpha \cdot \alpha^{2k+1}(1+\alpha^{-1}) \\
& = & (k+1)\alpha^{2k+2} + (k+2)\alpha^{2k+1}.
\end{eqnarray*}
This proves equation \eqref{eq-2-braids-odd} for $k+1$. Hence, equations \eqref{eq-2-braids-even} and \eqref{eq-2-braids-odd} are true for all $k\geq 0$.
\end{proof}

\end{document}

%% file: B+.tex
\setlength{\unitlength}{1pt}
\begin{picture}(60,55)(-30,-15)

\put(-20,0){\vector(1,1){40}}

\put(20,0){\line(-1,1){18}}

\put(-2,22){\vector(-1,1){18}}

\put(-4,-15){$B_+$}

\end{picture}

%% file: B-.tex
\setlength{\unitlength}{1pt}
\begin{picture}(60,55)(-30,-15)

\put(20,0){\vector(-1,1){40}}

\put(-20,0){\line(1,1){18}}

\put(2,22){\vector(1,1){18}}

\put(-4,-15){$B_-$}

\end{picture}

%% file: B0.tex
\setlength{\unitlength}{1pt}
\begin{picture}(60,55)(-30,-15)

\put(-20,0){\vector(0,1){40}}

\put(20,0){\vector(0,1){40}}

\put(-4,-15){$B_0$}

\end{picture}

%% file: vertex.tex
\setlength{\unitlength}{1pt}
\begin{picture}(60,40)(-30,0)

\put(-20,0){\vector(1,1){20}}

\put(0,20){\vector(1,1){20}}

\put(20,0){\vector(-1,1){20}}

\put(0,20){\vector(-1,1){20}}

\end{picture}

%% file: arc_marked.tex
\setlength{\unitlength}{1pt}
\begin{picture}(60,45)(-30,-15)

\put(-20,15){\vector(1,0){40}}

\put(-30,15){\small{$x_i$}}

\put(23,15){\small{$x_s$}}

\put(-4,-15){$\Gamma_{i;s}$}

\end{picture}

%% file: wide-edge_marked.tex
\setlength{\unitlength}{1pt}
\begin{picture}(70,45)(-35,-15)

\put(-20,0){\vector(4,3){20}}

\put(-20,30){\vector(4,-3){20}}

\put(0,15){\vector(4,3){20}}

\put(0,15){\vector(4,-3){20}}

\put(-35,0){\small{$x_j$}}

\put(-35,25){\small{$x_i$}}

\put(28,0){\small{$x_t$}}

\put(28,25){\small{$x_s$}}

\put(-4,-15){$\Gamma_{i,j;s,t}$}

\end{picture}

%% file: crossing-1-1-res-0.tex
\setlength{\unitlength}{1pt}
\begin{picture}(60,55)(-30,-15)

\put(-20,0){\vector(0,1){40}}

\put(20,0){\vector(0,1){40}}

\put(-30,35){\small{$x_s$}}

\put(-30,0){\small{$x_i$}}

\put(23,35){\small{$x_t$}}

\put(23,0){\small{$x_j$}}

\put(-4,-15){$\Gamma_0$}

\end{picture}

%% file: crossing-1-1-res-1.tex
\setlength{\unitlength}{1pt}
\begin{picture}(60,55)(-30,-15)

\put(-20,0){\vector(1,1){20}}

\put(20,0){\vector(-1,1){20}}

\put(0,20){\vector(-1,1){20}}

\put(0,20){\vector(1,1){20}}

\put(-30,35){\small{$x_s$}}

\put(-30,0){\small{$x_i$}}

\put(23,35){\small{$x_t$}}

\put(23,0){\small{$x_j$}}

\put(-4,-15){$\Gamma_1$}

\end{picture}

%% file: arc.tex
\setlength{\unitlength}{1pt}
\begin{picture}(60,55)(-30,-15)

\put(0,0){\vector(0,1){40}}

\put(-10,35){\small{$x_s$}}

\put(-10,0){\small{$x_i$}}

\put(-4,-15){$A$}

\end{picture}

%% file: crossing+.tex
\setlength{\unitlength}{1pt}
\begin{picture}(60,55)(-30,-15)

\put(-20,0){\vector(1,1){40}}

\put(20,0){\line(-1,1){18}}

\put(-2,22){\vector(-1,1){18}}

\put(-30,35){\small{$x_s$}}

\put(-30,0){\small{$x_i$}}

\put(23,35){\small{$x_t$}}

\put(23,0){\small{$x_j$}}

\put(-4,-15){$C_+$}

\end{picture}

%% file: crossing-.tex
\setlength{\unitlength}{1pt}
\begin{picture}(60,55)(-30,-15)

\put(20,0){\vector(-1,1){40}}

\put(-20,0){\line(1,1){18}}

\put(2,22){\vector(1,1){18}}

\put(-30,35){\small{$x_s$}}

\put(-30,0){\small{$x_i$}}

\put(23,35){\small{$x_t$}}

\put(23,0){\small{$x_j$}}

\put(-4,-15){$C_-$}

\end{picture}

%% file: Conway+.tex
\setlength{\unitlength}{1pt}
\begin{picture}(80,30)(-40,-10)

\put(0,0){\vector(1,1){10}}
\put(10,10){\line(1,1){10}}

\put(0,0){\vector(-1,1){10}}
\put(-10,10){\line(-1,1){10}}

\put(0,0){\circle*{2}}

\put(20,20){\circle*{2}}

\put(-20,20){\circle*{2}}

\put(-5,-10){\small{$B_+$}}

\put(-40,15){\small{$B_0$}}

\put(32,15){\small{$B_-$}}

\end{picture}

%% file: Conway-.tex
\setlength{\unitlength}{1pt}
\begin{picture}(80,30)(-40,-10)

\put(0,0){\vector(1,1){10}}
\put(10,10){\line(1,1){10}}

\put(0,0){\vector(-1,1){10}}
\put(-10,10){\line(-1,1){10}}

\put(0,0){\circle*{2}}

\put(20,20){\circle*{2}}

\put(-20,20){\circle*{2}}

\put(-5,-10){\small{$B_-$}}

\put(-40,15){\small{$B_0$}}

\put(32,15){\small{$B_+$}}

\end{picture}

%% file: Conway+-labeled.tex
\setlength{\unitlength}{1pt}
\begin{picture}(80,30)(-40,-10)

\put(0,0){\vector(1,1){10}}
\put(10,10){\line(1,1){10}}

\put(0,0){\vector(-1,1){10}}
\put(-10,10){\line(-1,1){10}}

\put(0,0){\circle*{2}}

\put(20,20){\circle*{2}}

\put(-20,20){\circle*{2}}

\put(-5,-10){\small{$B_+$}}

\put(-40,15){\small{$B_0$}}

\put(32,15){\small{$B_-$}}

\put(-21,6){\tiny{$\Delta_\alpha$}}

\put(13,6){\tiny{$S_\alpha^{-1}$}}

\end{picture}

%% file: Conway--labeled.tex
\setlength{\unitlength}{1pt}
\begin{picture}(80,30)(-40,-10)

\put(0,0){\vector(1,1){10}}
\put(10,10){\line(1,1){10}}

\put(0,0){\vector(-1,1){10}}
\put(-10,10){\line(-1,1){10}}

\put(0,0){\circle*{2}}

\put(20,20){\circle*{2}}

\put(-20,20){\circle*{2}}

\put(-5,-10){\small{$B_-$}}

\put(-40,15){\small{$B_0$}}

\put(32,15){\small{$B_+$}}

\put(-40,6){\tiny{$-S_\alpha\circ\Delta_\alpha$}}

\put(13,6){\tiny{$S_\alpha$}}

\end{picture}

%% file: Conway+-labeled-spec.tex
\setlength{\unitlength}{1pt}
\begin{picture}(80,30)(-40,-10)

\put(0,0){\vector(1,1){10}}
\put(10,10){\line(1,1){10}}

\put(0,0){\vector(-1,1){10}}
\put(-10,10){\line(-1,1){10}}

\put(0,0){\circle*{2}}

\put(20,20){\circle*{2}}

\put(-20,20){\circle*{2}}

\put(-5,-10){\small{$B_+$}}

\put(-40,15){\small{$B_0$}}

\put(32,15){\small{$B_-$}}

\put(-21,6){\tiny{$\Delta$}}

\put(13,6){\tiny{$S^{-1}$}}

\end{picture}

%% file: Conway--labeled-spec.tex
\setlength{\unitlength}{1pt}
\begin{picture}(80,30)(-40,-10)

\put(0,0){\vector(1,1){10}}
\put(10,10){\line(1,1){10}}

\put(0,0){\vector(-1,1){10}}
\put(-10,10){\line(-1,1){10}}

\put(0,0){\circle*{2}}

\put(20,20){\circle*{2}}

\put(-20,20){\circle*{2}}

\put(-5,-10){\small{$B_-$}}

\put(-40,15){\small{$B_0$}}

\put(32,15){\small{$B_+$}}

\put(-40,6){\tiny{$-S\circ\Delta$}}

\put(13,6){\tiny{$S$}}

\end{picture}